\documentclass[reqno]{amsart}
\usepackage{fourier}
\usepackage{pdfsync}
\usepackage{amsfonts,amssymb,pseudocode}
\usepackage{verbatim,graphicx,enumerate,xspace,url}
\usepackage[colorlinks=true, pdfstartview=FitV, linkcolor=blue,
 citecolor=blue, urlcolor=red]{hyperref}

\usepackage{xr-hyper}
\usepackage{hyperref}

\renewcommand{\P}{{\mathbf P}}

\newcommand{\tmix}{t_{{\rm mix}}}

\newcommand{\ep}{\varepsilon}

\newcommand{\deq}{\stackrel{{\rm def}}{=}}

\newtheorem{thm}{Theorem}

\newtheorem{prop}[thm]{Proposition}

\newtheorem{lem}[thm]{Lemma}

\newtheorem*{thm*}{Theorem}
\newtheorem*{cor*}{Corollary}

\theoremstyle{definition}

\theoremstyle{remark}
\newtheorem{rmk}{Remark}
\newtheorem*{rmk*}{Remark}

\newtheorem*{fact*}{Fact}
\newtheorem*{not*}{Notation}
\newtheorem*{claim*}{Claim}

\title{Estimating the Spectral Gap of a Reversible Markov
  Chain\\ from a Short Trajectory} 
\author{David A.\ Levin}
 \address{Department of Mathematics, Fenton Hall, University of Oregon 1222,
 Eugene, OR, 97403-1222.}
 \author{Yuval Peres}
 \address{Microsoft Research, 1 Microsoft Way, Redmond WA 98052}
\begin{document}

\begin{abstract}
  The spectral gap $\gamma$ of an ergodic and reversible Markov chain 
  is an important parameter measuring the asymptotic rate of
  convergence. 
  In applications, the transition matrix $P$ may be unknown, yet one
  sample of the chain up to a fixed time $t$ may be observed.
  Hsu, Kontorovich, and Szepesvari \cite{HKS}
  considered the problem of estimating $\gamma$ from this data.
  Let $\pi$ be the stationary distribution of $P$, and
  $\pi_\star = \min_x \pi(x)$.  They 
  showed that, if $t = \tilde{O}\bigl(\frac{1}{\gamma^3 \pi_\star}\bigr)$, then
  $\gamma$ can be estimated to within multiplicative constants with
  high probability.  
  They also proved that $\tilde{\Omega}\bigl(\frac{n}{\gamma}\bigr)$
  steps are required for precise estimation of $\gamma$.  We show that $\tilde{O}\bigl(\frac{1}{\gamma \pi_\star}\bigr)$ steps of the chain
  suffice  to estimate
  $\gamma$ up to multiplicative constants with high probability.
  When $\pi$ is uniform, this matches (up to logarithmic corrections)
  the lower bound in \cite{HKS}.   
\end{abstract}
\maketitle

\section{Introduction}

Consider an ergodic and reversible Markov chain $\{X_s\}$ on a finite
state space of size $n$, with transition matrix $P$ and 
stationary distribution $\pi$.  We will assume that $P$ is positive
definite, to avoid complications arising from eigenvalues close to
$-1$.   The $\emph{spectral gap}$ of
the chain is $\gamma = 1 - \lambda_2$, where $\lambda_2$ is the second
largest eigenvalue of $P$.  The spectral gap is an
important parameter of intrinsic interest, as it governs the asymptotic rate of
convergence to stationarity. 



Suppose one does not know $P$, but is able to observe the chain
$\{X_s\}_{s=1}^t$ for $t$ steps.  Can one estimate $\gamma$ with precision from this data?    This question was
 studied by Hsu, Kantorovich, and Szepesvari in \cite{HKS}.   Their
 estimator is the spectral gap of the (suitably symmetrized)
 empirical transition matrix.   
They show that $t = \tilde{O}\bigl(\frac{1}{\pi_\star \gamma^3}\bigr)$
observations of the chain are enough to estimate $\gamma$ to within
a constant factor.\footnote{The notation $f = \tilde{O}(g)$
means that there is a universal constant $C$ and a polylog function $L$ of the
parameters such that $f \leq C \cdot L \cdot g$.  Similarly, $f = \tilde{\Omega}(g)$ means that
there is a universal constant $c$ and a polylog function $\ell$ of
the parameters such that $f \geq c \cdot \ell \cdot g$.}   See Theorem \ref{Thm:HKS} for a precise statement.
In the case where $\pi$ is uniform, the authors of \cite{HKS} also show that
$\tilde{\Omega}(\frac{1}{\pi_\star \gamma})$ steps are needed
to estimate $\gamma$. 
Here we show that $t = \tilde{O}(\frac{1}{\pi_\star\gamma} )$ is a sufficient 
number of observations to estimate $\gamma$ to within a 
constant factor.   In particular, we prove:

\begin{thm} \label{thm:simplemain}
  
  Fix $\delta > 0$.  There is an estimator $\tilde{\gamma}$ of $\gamma$ based on
  $\{X_s\}_{s=0}^t$, and 
  a polynomial function ${\mathcal L}$ of the
  logarithms of $\gamma^{-1},\pi_\star^{-1},\delta^{-1}$, and $n$,
  such that, if $t > \frac{1}{\pi_\star \gamma \ep^2} {\mathcal L}$,
  then we have 
  $\Bigl|\frac{\tilde{\gamma}}{\gamma} - 1\Bigr| < \ep$
  with probability at least $1-\delta$.
   
\end{thm}
The definition of ${\mathcal L}$ is in
\eqref{Eq:LDef}.

The proof of Theorem~\ref{thm:simplemain} applies the estimator of
Hsu, Kontorovich,  and Szepesvari to estimate the gap $\gamma_A$ of the
``skipped chain'' $\{X_{As}\}_{s=1}^{t/A}$.   By successively doubling
$A$,  with high probability one
can identify the first value $A$ such that $\gamma_A$ is uniformly
bounded below.   Once this $A$ is identified, the estimate of
$\gamma_A$ can be transformed to an estimate of $\gamma$.

While $\gamma$ is a parameters of intrinsic
interest, it is also related to another important parameter,
the \emph{mixing time}.  The mixing time 
$\tmix(\ep)$ is the first time such that (from every starting state)
the distribution of the chain is within $\ep$ from $\pi$ in
total-variation.
Always $\gamma^{-1} \leq \tmix(1/4) + 1$, however, if $\pi_\star = \min_x \pi(x)$,
then  $\tmix(\ep) \leq |\log(\ep \pi_\star)|\cdot\gamma^{-1}$.
See \cite{LPW} for background on the spectral gap and the
mixing time.

\section{Proof of Theorem \ref{thm:simplemain}}

We will repeatedly apply the following estimate of Hsu, 
Kontorovich, and Szepesvari:
\begin{thm}[{\cite{HKS}}] \label{Thm:HKS}
  There is an estimator $\hat{\gamma}$ of $\gamma$, based on
  $t$-steps of the Markov chain, such that for some absolute constant
  $C$, with probability at least $1-\delta$,
\begin{equation}\label{Eq:HKS}
  |\hat{\gamma} - \gamma|
  \leq C \Biggl( \sqrt{ \frac{\log\bigl( \frac{n}{\delta} \bigr) \cdot
      \log\bigl( \frac{t}{\pi_\star \delta} \bigr)}{\pi_\star \gamma t }} +
      \frac{\log(1/\gamma)}{\gamma t } \Biggr)
  := M(t; \delta, \pi, \gamma)  \,. 
\end{equation}
\end{thm}
Thus, Theorem \ref{Thm:HKS} says $t=\tilde{O}(\frac{1}{\pi_\star
  \gamma^3})$ steps suffice for $\hat{\gamma}/\gamma$ to be near $1$.


  We call the estimator $\hat{\gamma}$ the HKS
estimator.   Note that if 
\[
t_1 = t_1(\ep; \delta, \gamma) :=
\frac{1}{\pi_\star \gamma} \frac{12 C^2 \log(n/\delta) \log(12 C^2
    /(\ep^2 \pi_\star^2 \gamma \delta))}{ \ep^2} 
\]
then
\[
M(t_1; \delta, \pi, \gamma) 
\leq  \frac{\ep}{2} \sqrt{
\frac{ \log(\frac{12 C^2}{\ep^2 \pi_\star^2 \gamma \delta})
+ \log(\log(n/\delta)) + \log\log(\frac{12 C^2}{\ep^2 \pi_\star^2 \gamma
  \delta})}{
3\log( \frac{12 C^2}{\ep^2 \pi_\star^2 \gamma \delta})}}
+ \frac{\ep}{2} \leq \ep.
\]
(Each term in the numerator under the radical is at most a third of
the 
denominator.   We have used that $\pi_\star \leq 1/n$ in comparing the
second term in the numerator to the denominator.)

For $a > 0$, the gap of the chain with transition matrix $P^{a}$
is denoted by $\gamma_a$, and the HKS estimator of $\gamma_a$,
based on $t/a$ steps of $P^a$, is denoted by $\widehat{\gamma_a}$.
Note that
\[
\gamma_a = 1 - (1-\gamma)^{a}\,.
\]

Let $\delta_\gamma = \frac{\delta}{\lfloor \log_2(1/\gamma) \rfloor + 1}$.
\begin{prop} \label{Prop:gammaA}
  Fix $\delta > 0$ and $\ep < 0.01$.  If $t > t_1(\ep/2;
  \delta_\gamma, \gamma)$,
  then there is an integer-valued random variable $A$, based on $t$
  steps of the Markov chain, and an event $G(\ep)$ having probability at
  least $1-\delta$, such that on $G(\ep)$,
  \begin{align*}
    0.30 < \gamma_A  < 0.54 & \quad  \text{if } \gamma < 1/2\,,\\
    A = 1 & \quad \text{if } \gamma \geq 1/2 \,.
  \end{align*}
  Moreover, on $G$, the HKS estimator $\widehat{\gamma_A}$ applied
  to the chain $\{X_{As}\}_{s=0}^{t/A}$ satisfies
  \[
    | \widehat{\gamma_A} - \gamma_A | < \ep \,. 
  \]
\end{prop}

Define 
\begin{equation} \label{Eq:EstSucc}
  G(a; \ep) = \bigl\{ |\gamma_a - \widehat{\gamma_a}| < \ep \bigr\}\,.
\end{equation}
\begin{lem}  \label{Lem:agamma}
  Fix $t \geq t_1(\ep/2 ; \delta, \gamma)$.
  If $a\gamma \leq 1$, then $\P(G(a;\ep)) > 1 - \delta$.
\end{lem}
\begin{proof}
Recall the bound $M(t;\delta,\pi,\gamma)$ on the right-hand side
of \eqref{Eq:HKS}.  If $\gamma_a \geq \frac{1}{2} \gamma a$,
then
\[
M(t/a; \delta_\gamma, \pi, \gamma_a) \leq
2 M(t; \delta, \pi, \gamma) \leq 2 \frac{\ep}{2} 
=\ep \,,
\]
and the lemma follows from applying Theorem \ref{Thm:HKS} to
the $P^a$-chain.  We now show that $\gamma_a \geq \frac{1}{2} \gamma
a$. 
Expanding $(1-\gamma)^a$, there exists $\xi \in [0, a^{-1}]$ 
such that
\begin{equation*}
  \gamma_a  = 1 - (1-\gamma)^a 
  = \gamma a - \frac{a(a-1)(1-\xi)^{a-2}\gamma^2}{2} 
  \geq \frac{\gamma a}{2} \,.
\end{equation*}
(We have used the hypothesis $a \gamma \leq 1$ in the inequality.)
\end{proof}

\begin{proof}[Proof of Proposition \ref{Prop:gammaA}]
Let $\delta_\gamma = \frac{\delta}{\lfloor \log_2(1/\gamma)\rfloor + 1}$. 
Fix $t > t_1(\ep/2; \delta_\gamma, \gamma)$.

Set $K_\gamma := \Bigl\lfloor \log_2\Bigl(\frac{1}{\gamma}\Bigr)
\Bigr\rfloor$, and let $\{X_s\}_{s=1}^t$ be $t$ steps of the Markov chain.
Consider the following algorithm:

Begin by setting $k=0$.  Let $a = 2^k$.   Using the ``skipped'' chain
$\{X_{as}\}_{s=1}^{t/A}$ observed for $t/A$ steps, form the HKS
estimator $\widehat{\gamma_a}$ of the spectral gap $\gamma_a$ of the
skipped chain.  If $\widehat{\gamma_A} > 0.31$ then set $A = a = 2^k$,
and stop.  Otherwise, increment $k$ and repeat.


Define the event $G = G(\ep)\deq \bigcap_{k=0}^{K_\gamma} G(2^k; \ep)$.
If $k \leq K_\gamma$, then $\gamma 2^k \leq \gamma 2^{\log_2(1/\gamma)} \leq 1$
and Lemma \ref{Lem:agamma} implies that
\[
\P(G^c) \leq \sum_{k=0}^{K_\gamma} \P(G(2^k; \ep)^c) \leq (K_\gamma+1)
 \frac{\delta}{K_\gamma + 1} = \delta \,.
\]

On $G$, if $\gamma \geq 1/2$, then 
$|\widehat{\gamma} - \gamma| < 0.01$, and
consequently $\hat{\gamma} \geq 0.49 > 0.31$.
In this case, $A=1$ on $G$.


On the event $G$, if the algorithm has not terminated by step $k-1$, 
then:
\begin{itemize}
  \item If $\gamma_{2^k} \leq 0.30$, then the algorithm does not
    terminate at step $k$;
  \item if $\gamma_{2^k} > 0.32$, then the algorithm terminates at
    step $k$.
\end{itemize}
Also, 
\[
\gamma_{2^{K_\gamma}} \geq 1 - (1-\gamma)^{\frac{1}{2\gamma}}
\geq  1 - e^{-1/2} \geq 0.39 \,,
\]
so the algorithm always terminates before $k = K_\gamma$ on $G$.

Finally, on $G$, if $A > 1$, then $\gamma_{A/2} \leq 0.32$, whence
\[
\gamma_A = 1 - (1-\gamma_{A/2})^2
\leq  1 - (0.68)^2 < 0.54 \,.
\]
If $\gamma < 1/2$ and $A = 1$, then  $\gamma_A = \gamma \leq 1/2$.
\end{proof}

\begin{proof}[Proof of Theorem \ref{thm:simplemain}]

For $C_0 = 23232\cdot C$, where $C$ is the constant in \eqref{Eq:HKS},
let
  \begin{equation} \label{Eq:t0}
  t_0(\ep; \delta, \gamma, \pi_\star) = t_0(\ep) := 
  \Biggl( \frac{1}{\pi_\star \gamma} \Biggr) \frac{C_0}{\ep^2 }
  {\mathcal L}
    \,,
  \end{equation}
where
\begin{equation} \label{Eq:LDef}
    {\mathcal L} = 
    \log\Bigl(\frac{C_0 \bigl( \bigl\lfloor \log_2(1/\gamma)
    \bigr\rfloor+ 1 \bigr)}{\ep^2 \pi^2_\star \gamma
    \delta} \Bigr) \log\Biggl(\frac{n (\bigl\lfloor \log_2(1/\gamma)
    \bigr\rfloor + 1 )}{\delta}\Biggr)  \,.
\end{equation}

Fix $t > t_0(\ep) = t_1(\ep/44; \delta_\gamma, \gamma)$.
Let $A$ and $G$ be as defined in Proposition
\ref{Prop:gammaA}.  Assume we are on the event $G = G(\ep/22)$ for the
rest of this proof.

Suppose first that $\gamma < 1/2$.
We have $0.30 < \gamma_A < 0.54$, and
\[
|\widehat{\gamma_A} - \gamma_A| < \frac{\ep}{22} < 0.01 \,,
\]
so both $\gamma_A$ and $\widehat{\gamma_A}$ are in $[0.29, 0.55]$,
say. 

Let $h(x) = 1 - (1-x)^{1/A}$, so $\gamma = h(\gamma_A)$.  Note that
on $[0.29,0.55]$,
\[
\frac{d}{dx} \log h(x) 
= \frac{1}{1 - (1-x)^{1/A}} \frac{1}{A} (1-x)^{1/A - 1}
\leq \frac{1}{A[1 - (1-0.29)^{1/A}]} \frac{1}{0.45} \,.
\]
Since $(1-x)^{1/A} \leq 1 - x/A + x^2/(2A^2)$,
\[
\frac{d}{dx} \log h(x)
\leq \frac{1}{(0.29 - 0.29^2/2)(0.45)} < 11 \,.
\]
Thus, $|\frac{d}{dx} \log h(x)|$ is bounded (by $11$) on $[0.29,0.55]$.
Write $\tilde{\gamma} = h(\widehat{\gamma_A})$.
We have
\[
|\log( h(\widehat{\gamma_A})/\gamma )| =
|\log h(\gamma_A) - \log h(\widehat{\gamma_A}) | \leq 11 
|\gamma_A - \widehat{\gamma_A}| \leq 11 \frac{\ep}{22}
\leq \frac{\ep}{2} \,.
\]
Thus,
\[
\frac{h(\widehat{\gamma_A})}{\gamma} \leq e^{ \ep/2} 
\leq 1 +  \ep \,.
\]
Similarly, $\frac {\gamma}{h(\widehat{\gamma_A})} \leq e^{\ep/2}$, so
\[
\frac{h(\widehat{\gamma_A})}{\gamma} \geq e^{- \ep/2} \geq 1 -  \ep
\,.
\]
Suppose that $\gamma \geq 1/2$.  Then $A = 1$ on the event
$G$, and 
\[
|\widehat{\gamma} - \gamma | < \frac{\ep}{22} \,,
\]
so
\[
\left| \frac{\widehat{\gamma}}{\gamma} - 1 \right| <
\frac{\ep}{22 \gamma} \leq  \ep \,.
\]
\end{proof}

\begin{rmk}
  If $t < t_0(\ep)$, then our estimation procedure is
  not guaranteed to produce a sensible estimate.
\end{rmk}

\section*{Acknowledgements}

We thank Jian Ding for helpful conversations on this topic.

\bibliographystyle{plain}

\end{document}